\documentclass[a4paper,twocolumn,english]{IEEEtran}

\usepackage[latin1]{inputenc}
\usepackage{amsmath}
\usepackage{graphicx}
\usepackage{amssymb}
\usepackage{amsthm}
\usepackage{cite}
\usepackage{array}
\usepackage[bb=boondox]{mathalfa}

\newtheorem{theorem}{Theorem}
\newtheorem{lemma}{Lemma}
\newtheorem{remark}{Remark}

\newtheorem{corollary}{Corollary}
\newtheorem{example}{Example}

\allowdisplaybreaks

\begin{document}

\title{Exponential stability of time-delay systems via new weighted integral inequalities}

\author{L.V. Hien and H. Trinh\thanks{{\em}}

\thanks{L.V. Hien is with the the School of Engineering,
Deakin University, VIC 3217, and the Department of Mathematics,
Hanoi National University of Education, Hanoi, Vietnam (e-mail: hienlv@hnue.edu.vn).}
\thanks{H. Trinh is with the School of Engineering,
Deakin University, VIC 3217, Australia (e-mail: hieu.trinh@deakin.edu.au).}}

\maketitle
\pagestyle{plain}

\begin{abstract}
In this paper, new weighted integral inequalities (WIIs) are first derived
by refining the Jensen single and double inequalities.
It is shown that the newly derived inequalities in this paper encompass both the
Jensen inequality and its most recent improvements based on Wirtinger integral inequality.
The potential capability of the proposed WIIs is demonstrated through applications in 
exponential stability analysis for some classes of time-delay systems
in the framework of linear matrix inequalities (LMIs).
The effectiveness and least conservativeness of the derived stability conditions using
WIIs are shown by various numerical examples.
\end{abstract}


\section{Introduction }

The problem of stability analysis and its applications to control of time-delay systems
is essential and of great importance for both theoretical and practical reasons \cite{Gu}.
This problem has attracted considerable attention during the last decade \cite{Sipa,Li,FC,HH}.
Many important results on asymptotic stability of time-delay systems have been established
using the Lyapunov-Krasovskii functional (LKF) method in the framework
of linear matrix inequalities (LMIs) \cite{SG}. It is a fact that asymptotic stability is a
synonym of exponential stability \cite{Mond}, and in many applications, it is important
to find estimates of the transient decaying rate of time-delay systems \cite{Xu06}.
Therefore, a great deal of efforts has been devoted to study exponential stability of
time-delay systems \cite{Mond,Xu06,Liu,RC,KP,Nam08,HP09,ZP,HP11,BNP,PKR,GGXH,Cao}.
To derive an estimate, also referred to as $\alpha$-stability, of the exponential convergence rate
of a time-delay system, various approaches have been proposed in the literature.
For example, state transformation $\xi(t)=e^{\alpha t}x(t)$
combined with the Lyapunov-Krasovskii functional method \cite{Xu06,Liu,RC,KP,Nam08},
model transformation \cite{HP09}, constructing modified LKFs with exponential
weighted functions \cite{Mond,HP11,BNP,PKR,GGXH,Cao}, estimating the
Lyapunov components \cite{ZP} or modified comparison principle \cite{FC,Ngoc}.

However, looking at the literature, it can be realized that the proposed methods in
the aforementioned works usually introduce conservatism in exponential stability 
conditions not only on the exponential convergence rate but also on the maximal allowable
 delay and the number of matrix variables. Therefore, aiming at reducing conservativeness
 of exponential stability conditions, ant important and relevant issue is to improve some integral-based inequalities.

In this paper, we first propose some new weighted integral inequalities (WIIs) which are suitable
 to use in exponential stability analysis for time-delay systems. We show that the
 newly derived inequalities in this paper encompass both the Jensen inequality \cite{GKC}
and some of its recent improvements based on Wirtinger integral inequality \cite{SG,Park}.
We then employ the proposed WIIs to derive new exponential stability
 conditions for some classes of time-delay systems in the framework of linear matrix inequalities.
 Numerical examples are provided in this paper to show the efficiency and potential capability of the newly derived WIIs.

The rest of this paper is organized as follows. In Section 2,
some preliminary results are presented.
New weighted integral inequalities and their applications in exponential stability analysis for
some classes of time-delay systems are presented in Section 3 and Section 4, respectively.
Section 5 provides numerical examples to demonstrate the effectiveness of the obtained results.
The paper ends with a conclusion and references.

\section{Preliminaries}

It can be realized in many contributions that, to derive the exponential estimates
for time-delay systems, a widely used approach is the use of weighted exponential
Lyapunov-Krasovskii functional \cite{Mond}.
For example, a functional of the form
\begin{equation}\label{e2.1}
V(x_t)=\int_{-\tau}^0\int_{t+s}^te^{\alpha(u-t)}\dot{x}^T(u)R\dot{x}(u)duds
\end{equation}
where $x$ is the state vector, scalars $\alpha>0,\tau>0$ and matrix $R>0$,
has been used in many works in the literature \cite{HP11,BNP,PKR,GGXH,Cao}.
The derivative of $V(x_t)$ is given by
\begin{equation}\label{e2.2}
\dot V(x_t)=\tau\dot{x}^T(t)R\dot{x}(t)-\int_{t-\tau}^te^{\alpha(s-t)}\dot{x}^T(s)R\dot{x}(s)ds.
\end{equation}
In order to generate LMIs conditions, an estimate on the second term of \eqref{e2.2} is obviously needed.
The problem raised here is how to find a tighter lower bound of a weighted
integral of quadratic terms in the following form
\[
I_w(\varphi,\alpha)=\int_a^be^{\alpha(s-b)}\varphi^T(s)R\varphi(s)ds
\]
where $\alpha>0$ is a scalar, $\varphi\in C([a,b],\mathbb{R}^n)$
and $R$ is a symmetric positive definite matrix in $\Bbb{R}^{n\times n}$,
$R\in\Bbb{S}^+_n$. When
$\alpha=0$ we write $I(\varphi)$ instead of $I_w(\varphi,0)$.

Inspired from the proof of the Jensen inequality \cite{GKC}, we have the following results
which referred in this paper to as Jensen-based weighted integral inequalities (WIIs)
in single and double forms.

\begin{lemma}\label{lm2.1}
For a given matrix $R\in\Bbb{S}^+_n$, scalars $b>a$, $\alpha>0$, and
a function $\varphi\in C([a,b],\mathbb{R}^n)$,
the following inequalities hold
\begin{equation}\label{e2.3}
I_w(\varphi,\alpha)\geq \frac{\alpha}{\gamma_0}\Big(\int_a^b\varphi(s)ds\Big)^TR\Big(\int_a^b\varphi(s)ds\Big),
\end{equation}
\begin{equation}\label{e2.4}
\begin{aligned}
&\int_a^b\int_s^be^{\alpha(u-b)}\varphi^T(u)R\varphi(u)duds\\
&\quad\geq \frac{\alpha^2}{\gamma_1}\Big(\int_a^b\int_s^b\varphi(u)duds\Big)^T
R\Big(\int_a^b\int_s^b\varphi(u)duds\Big),
\end{aligned}
\end{equation}
where $\gamma_k=e^{\alpha(b-a)}-\sum_{j=0}^k\frac{\alpha^j(b-a)^j}{j!}$, $k\geq 0$.
\end{lemma}

\begin{proof}
By taking integral of the inequality

$\begin{bmatrix} e^{\alpha(s-b)}\varphi^T(s)R\varphi(s)&\varphi^T(s)\\
\varphi(s)&e^{\alpha(b-s)}R^{-1}\end{bmatrix}\geq 0$
we obtain
\[
\begin{bmatrix} I_w(\varphi,\alpha)&\int_a^b\varphi^T(s)ds\\
\int_a^b\varphi(s)ds&\rho(\alpha)R^{-1}\end{bmatrix}\geq 0
\]
which implies \eqref{e2.3} by Schur complement.
The proof of \eqref{e2.4} is similar and thus it is omitted here.
\end{proof}

\begin{remark}\label{rm2.1}
Obviously $\frac{\alpha}{\gamma_0}>e^{-\alpha(b-a)}$
for all $\alpha>0, b>a$. Therefore, \eqref{e2.3} gives a new lower bound
in comparison to the common estimate $I_w(\varphi,\alpha)\geq e^{-\alpha(b-a)}I(\varphi)$.
\end{remark}

\begin{remark}\label{rm2.2}
When $\alpha$ approaches zero the previous inequalities
lead to the Jensen inequality in single and double form, respectively.
More precisely, from the fact that $\lim_{\alpha\to 0^+}\frac{\gamma_k}{\alpha^k}=\frac{(b-a)^k}{k!}$
we readily obtain the following results
\begin{equation}\label{e2.5}
\int_a^b\varphi^T(s)R\varphi(s)ds
\geq \frac{1}{b-a}\Big(\int_a^b\varphi(s)ds\Big)^TR\Big(\int_a^b\varphi(s)ds\Big),
\end{equation}
\begin{equation}\label{e2.6}
\begin{aligned}
&\int_a^b\int_s^b\varphi^T(u)R\varphi(u)duds\\
&\quad \geq \frac{2}{(b-a)^2}\Big(\int_a^b\int_s^b\varphi(u)duds\Big)^T
R\Big(\int_a^b\int_s^b\varphi(u)duds\Big).
\end{aligned}
\end{equation}
\end{remark}

\section{New weighted integral inequalities}\label{sec:3}

In this section, some new weighted integral inequalities are derived by refining \eqref{e2.3}, \eqref{e2.4}. 
In the following, let us denote
\[
J^g_w(\varphi,\alpha)=I_w(\varphi,\alpha)-\frac{\alpha}{\gamma_0}\Big(\int_a^b\varphi(s)ds\Big)^TR
\Big(\int_a^b\varphi(s)ds\Big)
\]
as the gap of \eqref{e2.3}. By refining \eqref{e2.3} we find a new lower bound for $J_w(\varphi,\alpha)$
other than zero. First, let us introduce the following notations for given scalars $b>a$, $\alpha>0$,
and $\varphi\in C([a,b],\Bbb{R}^n)$
\begin{align*}
& \ell=b-a,\quad A_\alpha=\frac{\gamma_0}{\alpha^2}-\frac{(1+\gamma_0)\ell^2}{\gamma_0},\\
&L_1=\begin{bmatrix}1&-\frac{\alpha\gamma_0}{\gamma_1}\end{bmatrix},\;
\zeta=\mathrm{col}\left\{\int_a^b\varphi(s)ds,\int_a^b\int_s^b\varphi(u)duds\right\}.
\end{align*}
By using the Taylor series expansion of exponential function, it can be verified that
$A_\alpha>0$ for any $\alpha>0$. We also use the notion of Kronecker product $A\otimes B$ for matrices
$A\in\Bbb{R}^{n\times m}, B\in\Bbb{R}^{q\times r}$.
For more details about the Kronecker product, we refer the readers to \cite{SH}.

\begin{lemma}\label{lm3.1}
For a given $n\times n$ matrix $R>0$, scalar $\alpha>0$, and
a function $\varphi\in C([a,b],\mathbb{R}^n)$, the following inequality holds
\begin{equation}\label{e3.1}
J^g_w(\varphi,\alpha)\geq \frac{\alpha}{\rho_0}\zeta^T(L_1^TL_1\otimes R)\zeta
\end{equation}
where $\rho_0=\left(\frac{\alpha\gamma_0}{\gamma_1}\right)^2A_\alpha
=\frac{\gamma_0}{\gamma_1^2}\left[\gamma_0^2-(\alpha\ell)^2e^{\alpha\ell}\right]$.
\end{lemma}

\begin{proof} 
For any $\varphi\in C([a,b],\mathbb{R}^n)$, we define an approximation
function $\psi\in C([a,b],\mathbb{R}^n)$ as follows
\begin{equation}\label{e3.2}
\psi(t)=\varphi(t)-\frac{\alpha e^{\alpha(b-t)}}{\gamma_0}\int_a^b\varphi(s)ds+h(t)\chi
\end{equation}
where $h(t)$ is a real valued function on $[a,b]$ and $\chi\in\mathbb{R}^n$
 is a constant vector which will be defined later.
For brevity we let $w(t)=e^{\alpha(b-t)}$ and predefine $h(t)=w(t)p(t)$,
where $p(t)$ belongs to $\Bbb{P}_k$,
the set of polynomials of order less than $k$. A prior computation gives
\begin{align}
J^g_w(\psi,\alpha)&=J^g_w(\varphi,\alpha)+J_w(h)\chi^TR\chi\notag\\
&-2\frac{\alpha}{\gamma_0}\int_a^bh(s)ds\chi^TR\int_a^b\varphi(s)ds\notag\\
&+2\chi^TR\bigg(p(a)\int_a^b\varphi(s)ds+p'(a)\int_a^b\int_s^b\varphi(u)duds\notag\\
&+\int_a^bp''(s)\int_s^b\int_u^b\varphi(v)dvduds\bigg),\label{e3.3}
\end{align}
where $J_w(h)=\Big[\int_a^bw^{-1}(s)h^2(s)ds-\frac{\alpha}{\gamma_0}\Big(\int_a^bh(s)ds\Big)^2\Big]$.

Now, for any $p\in\Bbb{P}_1$ which we can write $p(t)=c_0+c_1t$, $c_1\neq 0$.
Then
\begin{align*}
\int_a^bh(s)ds&=\frac{p(a)e^{\alpha\ell}-p(b)}{\alpha}+\frac{c_1\gamma_0}{\alpha^2},\\
\int_a^bw^{-1}(s)h^2(s)ds&=\frac{e^{\alpha\ell}p^2(a)-p^2(b)}{\alpha}\\
&\quad+\frac{2c_1\left(e^{\alpha\ell}p(a)-p(b)\right)}{\alpha^2}+\frac{2c_1^2\gamma_0}{\alpha^3}
\end{align*}
and thus $J_w(h)=\frac{A_\alpha}{\alpha}c_1^2$.
From \eqref{e3.3} we obtain
\begin{equation}\label{e3.4}
J^g_w(\psi,\alpha)=J^g_w(\varphi,\alpha)+\frac{A_\alpha c_1^2}{\alpha}\chi^TR\chi
-\frac{2\gamma_1c_1}{\alpha\gamma_0}\chi^TR(L_1\otimes I_n)\zeta.
\end{equation}
By Lemma \ref{lm2.1} $J^g_w(\psi,\alpha)\geq 0$ which leads to
\begin{equation}\label{e3.5}
J^g_w(\varphi,\alpha)\geq -\frac{A_\alpha c_1^2}{\alpha}\chi^TR\chi
+\frac{2\gamma_1c_1}{\alpha\gamma_0}\chi^TR(L_1\otimes I_n)\zeta.
\end{equation}
Hereafter, we will denote by $\mathcal{R}(J^g_w(\varphi,\alpha))$ the right-hand side of \eqref{e3.5}.
Now we define vector $\chi$ in the form
$\chi=\frac{\lambda}{c_1}(L_1\otimes I_n)\zeta$, where $\lambda$ is a scalar, then
\[\mathcal{R}(J^g_w(\varphi,\alpha)) =\frac1{\alpha}
\Big(2\frac{\gamma_1}{\gamma_0}\lambda-A_\alpha\lambda^2\Big)\zeta^T(L_1^TL_1\otimes R)\zeta.
\]
The function $f(\lambda)=2\frac{\gamma_1}{\gamma_0}\lambda-A_\alpha\lambda^2$
attains its maximum $\frac{\gamma_1^2}{A_\alpha\gamma_0^2}$
at $\lambda=\frac{\gamma_1}{A_\alpha\gamma_0}$. Then it follows from \eqref{e3.5} that
$J^g_w(\varphi,\alpha)\geq \frac{\alpha\gamma_1^2}{A_\alpha(\alpha\gamma_0)^2}
\zeta^T(L_1^TL_1\otimes R)\zeta$ which completes the proof.
\end{proof}

\begin{remark}\label{rm3.1}
It is interesting that estimate \eqref{e3.1}
does not depend on the selection of first-order polynomial $p\in\Bbb{P}_1$.
In other words, inequality \eqref{e3.1} can be derived from \eqref{e3.2}
for any function $h(t)$ of the form $(c_0+c_1t)e^{\alpha(b-t)}$, $c_1\neq 0$.
Of course, when $c_1=0$ then \eqref{e3.2} leads to \eqref{e2.3}.
\end{remark}

\begin{remark}\label{rm3.2}
By repeating the proof of Lemma \ref{lm3.1} with the approximation
\begin{equation}\label{e3.6}
\psi(t)=\varphi(t)-\frac{\alpha^2w(t)}{\gamma_1}\int_a^b\int_s^b\varphi(u)duds+w(t)p(t)\chi
\end{equation}
where $w(t)=e^{\alpha(b-t)}$ and $p\in\Bbb{P}_1$ then \eqref{e2.4}
leads to double WII formulated in the following lemma.
\end{remark}

\begin{lemma}\label{lm3.2}
For a given $n\times n$ matrix $R>0$, scalar $\alpha>0$, and
a function $\varphi\in C([a,b],\mathbb{R}^n)$, the following inequality holds
\begin{align}
&\int_a^b\int_s^be^{\alpha(u-b)}\varphi^T(u)R\varphi(u)duds\\
&\qquad \geq \frac{\alpha^2}{\gamma_1}\hat{\zeta}^T(L_0^TL_0\otimes R)\hat{\zeta}
+\frac{4\alpha^2}{\rho_1}\hat{\zeta}^T(L_2^TL_2\otimes R)\hat{\zeta}\label{e3.7}
\end{align}
where $\hat{\zeta}=\mathrm{col}\left\{\int_a^b\int_s^b\varphi(u)duds,
\int_a^b\int_s^b\int_u^b\varphi(v)dvduds\right\}$,
$L_0=[1\quad 0]$, $L_2=\begin{bmatrix}1&-\frac{\alpha\gamma_1}{\gamma_2}\end{bmatrix}$,
$\rho_1=\left(\frac{\alpha\gamma_1}{\gamma_2}\right)^2B_\alpha$
and $B_\alpha=\frac{2\gamma_1}{\alpha^2}-\frac{[\alpha\ell+(\alpha\ell-1)\gamma_0]\ell^2}{\gamma_1}$.
\end{lemma}

\begin{remark}\label{rm3.3}
The following facts can be found by using Taylor series of the exponential function
\begin{align*}
&\lim_{\alpha\to 0^+}\frac{\alpha}{\rho_0}=\frac{3}{b-a},\;
\lim_{\alpha\to 0^+}L_1=\hat{L}_1=\begin{bmatrix}1&\frac{-2}{b-a}\end{bmatrix},\\
&\lim_{\alpha\to 0^+}\frac{\alpha^2}{\rho_1}=\frac{4}{(b-a)^2},\;
\lim_{\alpha\to 0^+}L_2=\hat{L}_2=\begin{bmatrix}1&\frac{-3}{b-a}\end{bmatrix}.
\end{align*}

Therefore, when $\alpha$ approaches zero we obtain the following results which are the same
as those derived by the Wirtinger inequality in single and double form \cite{SG,Park}
\begin{equation}\label{e3.8}
\int_a^b\varphi^T(s)R\varphi(s)ds
\geq \frac{1}{b-a}\zeta^T(L_0^TL_0+3\hat{L}_1^T\hat{L}_1)\otimes R\zeta
\end{equation}
\begin{align}
\int_a^b\int_s^b&\varphi^T(u)R\varphi(u)duds\notag\\
&\geq \frac{2}{(b-a)^2}\hat{\zeta}^T(L_0^TL_0+8\hat{L}_2^T\hat{L}_2)\otimes R\hat{\zeta}.\label{e3.9}
\end{align}
\end{remark}

\begin{remark}\label{rm3.4}
The results obtained in Lemma \ref{lm3.1} and Lemma \ref{lm3.2} can be extended using \eqref{e3.2} and \eqref{e3.6}
where $h(t)=e^{\alpha(b-t)}p(t)$ and $p(t)$ belongs to the set of higher order polynomials, and then
new lower bounds for \eqref{e3.1}, \eqref{e3.7} can be derived. This will be addressed in future works.
\end{remark}

\section{Exponential estimates for time-delay systems}\label{sec:4}

This section aims to demonstrate the effectiveness of the
newly weighted integral inequalities proposed in this paper
through applications to exponential stability analysis for two classes of time-delay systems.

\subsection{Systems with discrete and distributed constant delays}

Consider the following time-delay system
\begin{equation}\label{e4.1}
\begin{aligned}
&\dot x(t)=A_0x(t)+A_1x(t-h)+A_2\int_{t-h}^tx(s)ds,\; t\geq 0, \\
&x(t)=\phi(t),\; t\in[-h,0],
\end{aligned}
\end{equation}
where $A_0,A_1,A_2\in\Bbb{R}^{n\times n}$ are given constant matrices, $h\geq 0$ is known time-delay,
$\phi\in C([-h,0],\Bbb{R}^n)$ is the initial condition.

We recall here that \cite{Mond,HP09}, for a given $\sigma>0$, system \eqref{e4.1} is exponentially stable
with convergence rate $\sigma$ if there exists $\beta>0$ such that any solution $x(t,\phi)$ of \eqref{e4.1} satisfies
$\|x(t,\phi)\|\leq \beta\|\phi\|e^{-\sigma t},\; \forall t\geq 0$.

Let $e_i\in\Bbb{R}^{n\times 4n}$ defined by
$e_i=[0_{n\times(i-1)n}\; I_n \; 0_{n\times(4-i)n}]$, $i=1,\ldots,4$.
We denote $\mathcal{A}=A_0e_1+A_1e_2+A_2e_3$ and the following matrices
\begin{align*}
&F_0=\begin{bmatrix}e_1\\ e_3\\ e_4\end{bmatrix},\quad 
F_1=\begin{bmatrix}\mathcal{A}\\ e_1-e_2\\ he_1-e_3\end{bmatrix},\\
&F_2=\begin{bmatrix}e_1-e_2\\ he_1-e_3\end{bmatrix},\quad
F_3=\begin{bmatrix}he_1-e_3\\ h^2/2e_1-e_4\end{bmatrix},\\
&\Pi_0=F_0^TPF_1+F_1^TPF_0+\alpha F_0^TPF_0,\\
&\Pi_1=e_1^TQe_1-e^{-\alpha h}e_2^TQe_2+\mathcal{A}^T(hR+h^2/2Z)\mathcal{A},\\
&\Pi_2=\frac{\alpha}{\gamma_0}(e_1-e_2)^TR(e_1-e_2)+\frac{\alpha}{\rho_0}F_2^T(L_1^TL_1\otimes R)F_2,\\
&\Pi_3=\frac{\alpha^2}{\gamma_1}(he_1-e_3)^TZ(he_1-e_3)
+\frac{4\alpha^2}{\rho_1}F_3^T(L_2^TL_2\otimes Z)F_3
 \end{align*}
where $L_1,L_2$ and $\gamma_0,\gamma_1,\rho_0,\rho_1$ are defined in
\eqref{e2.3}, \eqref{e2.4}, \eqref{e3.1}, \eqref{e3.7} with $a=-h, b=0$.

\begin{theorem}\label{thm4.1}
Assume that, for a given $\alpha>0$, there exist symmetric positive definite matrices $P\in\Bbb{R}^{3n\times 3n}$,
$Q,R,Z\in\Bbb{R}^{n\times n}$ satisfying the following LMI
\begin{equation}\label{e4.2}
\Pi_0+\Pi_1-\Pi_2-\Pi_3<0.
\end{equation}
Then system \eqref{e4.1} is exponentially stable with a convergence rate $\sigma=\alpha/2$.
\end{theorem}

\begin{proof}
Consider the following Lyapunov-Krasovskii functional
\begin{equation}\label{e4.3}
\begin{aligned}
V(x_t)=\; &\tilde{x}^T(t)P\tilde{x}(t)+\int_{t-h}^te^{\alpha(s-t)}x^T(s)Qx(s)ds\\
&+\int_{-h}^0\int_{t+s}^te^{\alpha(u-t)}\dot{x}^T(u)R\dot{x}(u)duds\\
&+\int_{-h}^0\int_s^0\int_{t+u}^te^{\alpha(\theta-t)}\dot{x}^T(\theta)Z\dot{x}(\theta)duds
\end{aligned}
\end{equation}
where $\tilde{x}(t)=\mathrm{col}\left\{x(t),\int_{t-h}^tx(s)ds,\int_{t-h}^t\int_s^tx(u)duds\right\}$.

It follows from \eqref{e4.3} that $V(x_t)\geq\lambda_{\min}(P)\|x(t)\|^2$.
Taking derivative of $V(x_t)$ along trajectories of \eqref{e4.1} we obtain
\begin{equation}\label{e4.4}
\begin{aligned}
\dot V(x_t)&+\alpha V(x_t)
=\chi^T(t)\big(\Pi_0+\Pi_1\big)\chi(t)\\
&-\int_{t-h}^te^{\alpha(s-t)}\dot{x}^T(s)R\dot{x}(s)\\
&-\int_{-h}^0\int_{t+s}^te^{\alpha(u-t)}\dot{x}^T(u)Z\dot{x}(u)duds
\end{aligned}
\end{equation}
where\\
$\chi(t)=\mathrm{col}\{x(t),x(t-h), \int_{t-h}^tx(s)ds,\int_{-h}^0\int_{t+s}^tx(u)duds\}$.

By applying Lemma \ref{lm3.1} and Lemma \ref{lm3.2} to
the first and the second terms in \eqref{e4.4}, respectively, we then obtain
\begin{equation}\label{e4.5}
\dot V(x_t)+\alpha V(x_t)\leq\chi^T(t)(\Pi_0+\Pi_1-\Pi_2-\Pi_3)\chi(t).
\end{equation}
It follows from \eqref{e4.2} and \eqref{e4.5} that
$\dot V(x_t)+\alpha V(x_t)\leq 0$,
which yields $V(x_t)\leq V(\phi)e^{-\alpha t}$.
This leads to $\|x(t,\phi)\|\leq\sqrt{\frac{V(\phi)}{\lambda_{\min}(P)}}e^{-\alpha/2 t}$.
The proof is completed.
\end{proof}

\begin{remark}\label{rm4.1}
Note that the exponential transformation $z(t)=e^{\sigma t}x(t),\sigma>0$,
in general, is not applicable to access the exponential stability of system \eqref{e4.1}
because it leads to a system with time-varying coefficients. When $A_2=0$, using
the aforementioned transformation, system \eqref{e4.1} becomes
\begin{equation}\label{e4.6}
\dot z(t)=(A_0+\sigma I_n)z(t)+e^{\sigma h}A_1z(t-h).
\end{equation}

In many works found in the literature, in order to get exponential estimates for
system \eqref{e4.1} (with $A_2=0$), it was transformed to \eqref{e4.6} first and then
asymptotic stability conditions for \eqref{e4.6} were proposed. However, this approach
usually produces conservatism in exponential stability conditions due to the fact that
the exponential stability of  \eqref{e4.1} (with $A_2=0$) is just equivalent to the
boundedness of \eqref{e4.6} which is less restrictive then asymptotic stability.
Differ from those, and as discussed in Section 2 in this paper, we here propose
an improved approach used in exponential stability analysis for time-delay systems
by employing our new weighted inequalities derived in Lemma \ref{lm3.1} and Lemma \ref{lm3.2}.
in this paper.
\end{remark}

\subsection{Systems with interval time-varying delay}

Consider a class of linear systems with interval time-varying delay of the form
\begin{equation}\label{e4.7}
\begin{cases}
\dot x(t)=Ax(t)+A_dx(t-h(t)),\; t\geq 0\\
x(t)=\phi(t),\; t\in[-h_2,0]
\end{cases}
\end{equation}
where $A,A_d\in\mathbb{R}^{n\times n}$ are given constant matrices,
$h(t)$ is time-varying delay satisfying $0\leq h_1\leq h(t)\leq h_2$, where $h_1, h_2$ are known
constants involving the upper and the lower bounds of time-varying delay.

Let $e_i=[0_{n\times(i-1)n}\;\; I_n\;\;  0_{n\times(7-i)n}]$, $i=1,2,\ldots, 7$.
We denote $\mathcal{A}=Ae_1+A_de_3$ and
\begin{align*}
&\chi_1(t)=\mathrm{col}\left\{\begin{bmatrix}x(t)\\ x(t-h_1)\\ x(t-h(t))\\ x(t-h_2)\end{bmatrix},
\begin{bmatrix}\frac1{h_1}\int_{t-h_1}^tx(s)ds\\ \frac1{h(t)-h_1}\int_{t-h(t)}^{t-h_1}x(s)ds\\
 \frac1{h_2-h(t)}\int_{t-h_2}^{t-h(t)}x(s)ds\end{bmatrix}\right\},\\
&\Upsilon(h)=\mathrm{col}\{e_1, h_1e_5,(h-h_1)e_6+(h_2-h)e_7\},\\
&\Upsilon_0=\mathrm{col}\{ \mathcal{A}, e_1-e_2,e_2-e_4\},\\
&\Upsilon_1=\mathrm{col}\{e_1-e_2,h_1(e_1-e_5)\}, \\
& \Upsilon_2=\mathrm{col}\{e_2-e_3,e_2+e_3-2e_6\},\\
&\Upsilon_3=\mathrm{col}\{e_3-e_4,e_3+e_4-2e_7\}, \quad  \Delta=[\Upsilon_2^T\; \Upsilon_3^T]^T,\\
&\Omega_0(h)=\Upsilon(h)^TP\Upsilon_0+\Upsilon_0^TP\Upsilon(h)+\alpha \Upsilon(h)^TP\Upsilon(h),\\
&\Omega_1=e_1^TQ_1e_1-e^{-\alpha h_1}e_2^TQ_1e_2\\
&\qquad +e^{-\alpha h_1}e_2^TQ_2e_2-e^{-\alpha h_2}e_4^TQ_2e_4,\\
&\Omega_2=\mathcal{A}^T\left(h_1^2R_1+h_{12}^2e^{\alpha h_1}R_2\right)\mathcal{A},\\
&\Omega_3=\frac{\alpha h_1}{\tilde{\gamma}_0}(e_1-e_2)^TR_1(e_1-e_2)
+\frac{\alpha h_1}{\tilde{\rho}_0}\Upsilon_1^T(\tilde{L}_1^T\tilde{L}_1\otimes R_1)\Upsilon_1,\\
&\tilde{\gamma}_0=e^{\alpha h_1}-1, \quad \tilde{\gamma}_1=\tilde{\gamma}_0-\alpha h_1, \quad
\tilde{L}_1=[1\quad -\alpha\tilde{\gamma}_0/\tilde{\gamma}_1],\\
&\tilde{\rho}_0=\frac{\tilde{\gamma}_0}{\tilde{\gamma}_1^2}
\left[\tilde{\gamma}_0^2-(\alpha h_1)^2e^{\alpha h_1}\right],\quad h_{12}=h_2-h_1.
\end{align*}

\begin{theorem}\label{thm4.2}
Assume that there exist symmetric positive definite matrices
$P\in\Bbb{R}^{3n\times 3n}$, $Q_i, R_i\in\Bbb{R}^{n\times n}$,
$i=1,2$, and a matrix $X\in\Bbb{R}^{2n\times 2n}$
such that the following LMIs hold for $h\in\{h_1,h_2\}$
\begin{align}
&\Pi=\begin{bmatrix} \tilde{R}_2 & X\\ *& \tilde{R}_2\end{bmatrix}\geq 0,\label{e4.8}\\
&\Omega(h)=\Omega_0(h)+\Omega_1+\Omega_2-\Omega_3-e^{-\alpha h_{12}}\Delta^T\Pi\Delta<0,\label{e4.9}
\end{align}
where $\tilde{R}_2=\mathrm{diag}\{R_2,3R_2\}$. 
Then system \eqref{e4.7} is exponentially stable with a convergence rate $\alpha/2$.
\end{theorem}

First, we need the following lemmas.

\begin{lemma}\label{lm4.1}
If $\Omega(h_1)<0$ and $\Omega(h_2)<0$ then $\Omega(h)<0$, $\forall h\in[h_1,h_2]$.
\end{lemma}

\begin{proof}
It is obvious that $\frac{d^2}{dh^2}\Omega(h)=2\alpha\Gamma^TP\Gamma\geq 0$,
where $\Gamma=[0_{7n\times 2n}\; (e_6-e_7)^T]^T$.
Therefore, $\Omega(h)$ is a convex quadratic function with respect to $h$.
This completes the proof.
\end{proof}

\begin{lemma}[Improved reciprocally convex combination \cite{SG2}]\label{lm4.2}
For given symmetric positive definite matrices $R_1\in\Bbb{R}^{n\times n}$,
$R_2\in\Bbb{R}^{m\times m}$, if there exists a matrix $X\in\mathbb{R}^{n\times m}$ such that
$\begin{bmatrix} R_1&X\\ * &R_2\end{bmatrix}\geq 0$ then the inequality
\[
\begin{bmatrix} \frac{1}{\delta}R_1 & 0\\ 0 &\frac{1}{1-\delta}R_2\end{bmatrix} \geq
\begin{bmatrix} R_1&X\\  * & R_2\end{bmatrix}
\]
holds for all $\delta\in (0,1)$.
\end{lemma}

\begin{proof} Inspired from \cite{SG2}, we now consider the following LKF
\begin{align}
V(x_t)=\; & \chi_0^T(t)P\chi_0(t)+\int_{t-h_1}^te^{\alpha(s-t)}x^T(s)Q_1x(s)ds\notag\\
&+\int_{t-h_2}^{t-h_1}e^{\alpha(s-t)}x^T(s)Q_2x(s)ds\label{e4.10}\\
&+h_1\int_{-h_1}^0\int_{t+s}^te^{\alpha(u-t)}\dot{x}^T(u)R_1\dot{x}(u)duds\notag\\
&+h_{12}\int_{-h_2}^{-h_1}\int_{t+s}^te^{\alpha(u+h_1-t)}\dot{x}^T(u)R_2\dot{x}(u)duds\notag
\end{align}
where $\chi_0(t)=\mathrm{col}\{x(t),\int_{t-h_1}^tx(s)ds,\int_{t-h_2}^{t-h_1}x(s)ds\}$.

It follows from \eqref{e4.10} that $V(x_t)\geq\lambda_{\min}(P)\|x(t)\|^2$. 
In regard to the fact $\chi_0(t)=G_0(h)\chi_1(t)$ and $\frac{d}{dt}\chi_0(t)=G_1\chi_1(t)$,
the derivative of \eqref{e4.10} along trajectories of \eqref{e4.7} gives 
\begin{equation}\label{e4.11}
\begin{aligned}
\dot V(x_t)&+\alpha V(x_t)=\chi_1^T(t)\left(\Omega_0(h)+\Omega_1+\Omega_2\right)\chi_1(t)\\
&-h_1\int_{t-h_1}^te^{\alpha(s-t)}\dot{x}^T(s)R_1\dot{x}(s)ds\\
&-h_{12}\int_{t-h_2}^{t-h_1}e^{\alpha(s+h_1-t)}\dot{x}^T(s)R_2\dot{x}(s)ds.
\end{aligned}
\end{equation}

By Lemma \ref{lm3.1} we have
\begin{equation}\label{e4.12}
-h_1\int_{t-h_1}^te^{\alpha(s-t)}\dot{x}^T(s)R_1\dot{x}(s)ds\leq -\chi_1^T(t)\Omega_3\chi_1(t).
\end{equation}

Next, by splitting
\begin{align*}
\int_{t-h_2}^{t-h_1}&e^{\alpha(s+h_1-t)}\dot{x}^T(s)R_2\dot{x}(s)ds\\
&=\int_{t-h(t)}^{t-h_1}e^{\alpha(s+h_1-t)}\dot{x}^T(s)R_2\dot{x}(s)ds\\
&+\int_{t-h_2}^{t-h(t)}e^{\alpha(s+h_1-t)}\dot{x}^T(s)R_2\dot{x}(s)ds
\end{align*}
the second integral term of \eqref{e4.11} can be bounded by \eqref{e3.8} and Lemma \ref{lm4.1} as follows
\begin{align*}
-h_{12}&\int_{t-h_2}^{t-h_1}e^{\alpha(s+h_1-t)}\dot{x}^T(s)R_2\dot{x}(s)ds\\
&\leq -\frac{h_{12}e^{-\alpha h_{12}}}{h(t)-h_1}\chi_2^T(t)\tilde{R}_2\chi_2(t)
-\frac{h_{12}e^{-\alpha h_{12}}}{h_2-h(t)}\chi_3^T(t)\tilde{R}_2\chi_3(t)\\
&=-e^{-\alpha h_{12}}\begin{bmatrix}\chi_2(t)\\\chi_3(t)\end{bmatrix}^T
\begin{bmatrix}\frac{h_{12}}{h(t)-h_1}\tilde{R}_2 &0\\ 0&\frac{h_{12}}{h_2-h(t)}\tilde{R}_2\end{bmatrix}
\begin{bmatrix}\chi_2(t)\\\chi_3(t)\end{bmatrix}\\
&\leq -e^{-\alpha h_{12}}\begin{bmatrix}\chi_2(t)\\\chi_3(t)\end{bmatrix}^T
\begin{bmatrix}\tilde{R}_2 &X\\ *&\tilde{R}_2\end{bmatrix}
\begin{bmatrix}\chi_2(t)\\\chi_3(t)\end{bmatrix}
\end{align*}
where $\chi_2(t)=\Upsilon_2\chi_1(t)$ and $\chi_3(t)=\Upsilon_3\chi_1(t)$.
Note that the previous inequality is still valid when $h(t)$ tends to $h_1$ or $h_2$.
This leads to
\begin{equation}\label{e4.13}
\begin{aligned}
-h_{12}\int_{t-h_2}^{t-h_1}&e^{\alpha(s+h_a-t)}\dot{x}^T(s)R_2\dot{x}(s)ds\\
&\leq -e^{-\alpha h_{12}}\chi_1^T(t)\Delta^T\Pi\Delta\chi_1(t).
\end{aligned}
\end{equation}

Combining \eqref{e4.11}-\eqref{e4.13} we then obtain
\begin{equation}\label{e4.14}
\dot V(x_t)+\alpha V(x_t)\leq \chi_1^T(t)\Omega(h)\chi_1(t).
\end{equation}
By Lemma \ref{lm4.1}, \eqref{e4.9} implies that $\Omega(h)<0$ for all $h\in[h_1,h_2]$. 
Therefore, if \eqref{e4.9} holds for $h=h_1$ and $h=h_2$ then, from \eqref{e4.14}, $\dot V(x_t)+\alpha V(x_t)\leq 0$
which concludes the exponential stability of \eqref{e4.7} with guaranteed decay rate $\sigma=\alpha/2$.
The proof is completed.
\end{proof}

\begin{remark}\label{rm4.2}
When $\alpha$ approaches zero, by Remark \ref{rm3.3} and Theorem \ref{thm4.2}, we obtain the same
asymptotic stability conditions for system \eqref{e4.7} derived from improved Wirtinger's inequality \cite{SG2}.
\end{remark}

 \begin{corollary}\label{cr4.1}
 System \eqref{e4.7} is asymptotically stable for any delay $h(t)\in[h_1,h_2]$
if there exist symmetric positive definite matrices
 $P\in\Bbb{R}^{3n\times 3n}$, $Q_i, R_i\in\Bbb{R}^{n\times n}$, $i=1,2$, 
and a matrix $X\in\Bbb{R}^{2n\times 2n}$
 satisfying \eqref{e4.8} and the following LMI holds for $h\in\{h_1,h_2\}$
 \begin{equation}\label{e4.15}
 \tilde{\Omega}_0(h)+\Phi_1-\Phi_2-\Delta^T\Pi\Delta<0,
 \end{equation}
where
\begin{align*}
\tilde{\Omega}_0(h)&=\Upsilon(h)^TP\Upsilon_0+\Upsilon_0^TP\Upsilon(h),\\
\Phi_1&=e_1^TQ_1e_1-e_2^TQ_1e_2+e_2^TQ_2e_2-e_4^TQ_2e_4\\
&\quad +\mathcal{A}^T\left(h_1^2R_1+h_{12}^2R_2\right)\mathcal{A},\\
\Phi_2&=\Upsilon_4^T\mathrm{diag}\{R_1,3R_1\}\Upsilon_4,\\
\Upsilon_4&=\mathrm{col}\{e_1-2_2,e_1+e_2-2e_5\}.
\end{align*}
 \end{corollary}

\section{Examples}\label{sec:5}

\begin{example}\rm
Consider the following system \cite{SG,Park}
\begin{equation}\label{e5.1}
\dot x(t)=\begin{bmatrix}0.2&0\\0.2&0.1\end{bmatrix}x(t)
+\begin{bmatrix}-1 &0\\ -1 &-1\end{bmatrix}\int_{t-h}^tx(s)ds.
\end{equation}

An eigenvalue analysis \cite{SG} shows that \eqref{e5.1} remains stable for a constant delay in the range $h\in[0.2,2.04]$.
By the Wirtinger-based inequality approach, Theorem 6 in \cite{SG} and Theorem 1 in \cite{Park} guarantee the
asymptotic stability of \eqref{e5.1} for $h$ in the interval $[0.2,1.877]$ and $[0.2,1.9504]$, respectively.
In Theorem \ref{thm4.1} we fix $\alpha$ at $0.0002$, it is found that \eqref{e4.2} is feasible for $h\in[0.2,1.9778]$.
This clearly shows a reduction of conservatism of Theorem \ref{thm4.1}.
\end{example}

\begin{example}\rm
Consider system \eqref{e4.1} with the matrices taken from the literature
\[A_0=\begin{bmatrix} 0 & 1\\ -2 & 0.1\end{bmatrix},\quad
A_1=\begin{bmatrix} 0 & 0\\ 1 & 0\end{bmatrix},\quad
A_2=\begin{bmatrix}0&0\\0&0\end{bmatrix}.
\]

It is surprising  that, for this system, the exponential stability criteria proposed in \cite{Mond,Xu06,Liu},
especially \cite{Cao}, are not feasible for any $h>0, \alpha>0$.
By using the notations given in \cite{FC},
we have $\mathcal{M}=A_0^M+|A_1|=\begin{bmatrix}0 & 1\\ 3&0.1\end{bmatrix}$.
Thus, for any $v=(v_1,v_2)^T\in\Bbb{R}^2$,
$v_1>0,v_2>0$, $\mathcal{M}v=\begin{bmatrix}v_2\\ 3v_1+v_2\end{bmatrix}>0$.
This shows that the stability conditions proposed in \cite{FC} and \cite{Ngoc} are not
applicable to this system. Note also that, since $Re(eig(A_0+A_1))=0.05>0$, the delay-free
system is unstable and the results to access exponential stability of this system are much more scarce.
By employing the Wirtinger-based integral inequality, a significant improvement
of the asymptotic stability conditions was provided in \cite{SG}. To compare with our approach, we apply
Theorem 6 in \cite{SG} to \eqref{e4.6}. The obtained results in Table 1 show that,
 thanks to our new weighted integral inequalities proposed
in Lemma \ref{lm3.1} and Lemma \ref{lm3.2}, Theorem \ref{thm4.1} in this paper gives significantly better results.
The simulation result presented in Figure 1 is taken with $h=1.6$, $\sigma=0.045$, 
incremental step $s=10^{-4}$ and initial condition
$\phi(t)=[2\quad -1]^T$. It can be seen that the state trajectory $z(t)=e^{\sigma t}x(t)$ is bounded, and thus,
$x(t)$ exponentially converges to the origin with decay rate $\sigma=0.045$.

 \begin{table*}[!htbp]\label{tb:1}
 \caption{Decay rate $\sigma$ for various $h$ in Example 5.2}
 \begin{tabular*}{\textwidth}{l @{\extracolsep{\fill}} lllllll}\hline
 $h$& 0.3 & 0.5 & 0.8 & 1.0 & 1.5&1.6& NoDv\\ \hline
  \cite{SG} &0.0971& 0.1905& 0.2936 & 0.2766 & 0.0175 &- & $3n^2+2n$\\
  Theorem \ref{thm4.1} & 0.0971&0.2095&0.4195&0.4978&0.1039&0.045 &$6n^2+3n$\\
  \hline
  \end{tabular*}
  {NoDv: Number of Decision variable}
  \end{table*}

\begin{figure}[!ht]
\centering
\includegraphics[width=0.5\textwidth]{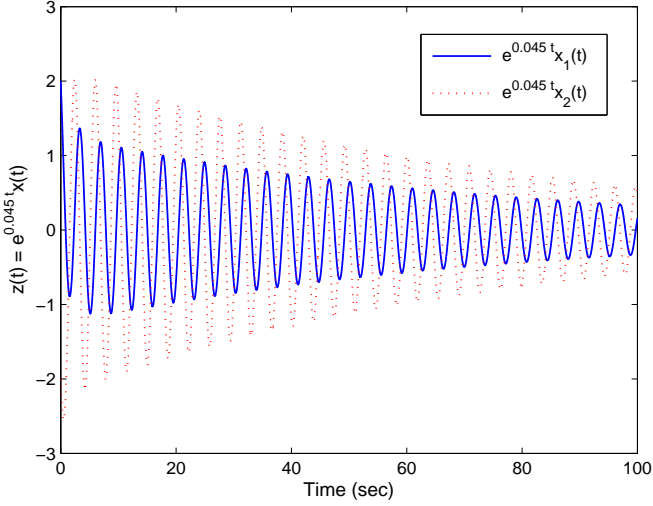}
\caption{Trajectory $z(t)=e^{0.045 t}x(t)$ with $h=1.6$.}\label{fig:1}
\end{figure}
\end{example}

\begin{example}\rm
Consider an active quarter-car suspension
system with control delay introduced in \cite{LJK}
\begin{equation}\label{e5.2}
\begin{cases}
\dot x(t)=Ax(t)+Bu(t-h(t)),\\
y(t)=Cx(t),
\end{cases}
\end{equation}
where $x(t)\in\Bbb{R}^4$ is the state, $u(t)$ is the control input,
$y(t)$ is the output and
\begin{align*}
&A=\begin{bmatrix}0 &0&1&-1\\ 0&0&0&1
\\ -\frac{k_s}{m_s}&0&-\frac{c_s}{m_s}&\frac{c_s}{m_s}\\
\frac{k_s}{m_u}&-\frac{k_t}{m_u}&\frac{c_s}{m_u}&-\frac{c_s+c_t}{m_u}\end{bmatrix},\\
&B=\begin{bmatrix}0\\0\\ \frac1{m_s}\\-\frac{1}{m_u}\end{bmatrix},\quad 
C=\begin{bmatrix} 1\\ 1 \\ 1\\ 0\end{bmatrix}^T.
\end{align*}

The following parameters are taken from \cite{LJK}.

 \begin{table}[!htbp]\label{tb2}
 \centering
 \caption{Quarter-car model parameters}
 \begin{tabular*}{0.48\textwidth}{@{\extracolsep{\fill}} cccccc}\hline
 $m_s$ & $m_u$ & $k_s$ & $k_t$ & $c_s$ & $c_t$ \\ \hline
 973kg &114kg &42720N/m &101115N/m &1095Ns/m & 14.6Ns/m\\
  \hline
  \end{tabular*}
 \end{table}

A static output feedback controller is proposed as $u(t)=Ky(t)$, where
$K$ is the controller gain. The closed-loop system is then in the form of \eqref{e4.7}
with $A_d=BKC$. For illustrative purpose we consider $K=1$.
It is noted first that, in this case, the exponential stability criteria proposed in
\cite{FC,Ngoc} based on positive system approach
also cannot access the exponential stability of the system.
 In \cite{GGXH}, some integral equalities were used to overcome the conservative estimates.
 However, when manipulating the derivative of the Lyapunov-Krasovskii functional,
all the integral terms were abandoned (see, Eq. (10) in \cite{GGXH}). This leads to the fact that the proposed
conditions in \cite{GGXH} are very conservative. We apply the main theorem in \cite{GGXH} to this example,
the obtained results for $h_1=1$ and various $h_2$ are listed in Table 3.
In \cite{ZP}, exponential convergence rate of solutions was derived
by estimating the maximal Lyapunov exponents. By Theorem 3 in \cite{ZP},
the exponential convergence rate $\sigma\in (0,\lambda_*]$, where $\lambda_*$ is the unique
positive solution of equation $\lambda+0.0087e^{\lambda h_2}=0.2707$.
We apply Theorem \ref{thm4.2} in this paper for $h_1=1$ and various $h_2$.
The obtained results are exposed in Table 3. Clearly a significant reduction of
conservatism is delivered by Theorem \ref{thm4.2}. This shows the effectiveness of our approach.
The simulation result presented in Figure 2 is taken with $h(t)=1+5|\sin(t)|$, $\sigma=0.2546$
and incremental step $s=10^{-4}$ which illustrates our theoretical results.

\begin{table*}[!htbp]\label{tb:3}
\caption{Decay rate $\sigma$ for $h_1=1$ and various $h_2$}
\begin{tabular*}{\textwidth}{l @{\extracolsep{\fill}} llllll}\hline
$h_2$& 2 & 3& 4 & 5 & 6& NoDv\\ \hline
\cite{GGXH} &0.1338&0.1316&0.1300&0.1290&0.1267&$50.5n^2+6.5n$\\
\cite{ZP} &0.2562&0.2522& 0.2473 & 0.2416 & 0.2351&- \\
Theorem \ref{thm4.2} & 0.2690&0.2672&0.2644&0.2603&0.2546& $10.5n^2+3.5n$\\
\hline
\end{tabular*}
\end{table*}
\end{example}

\begin{figure}[!ht]
\centering
\includegraphics[width=0.5\textwidth]{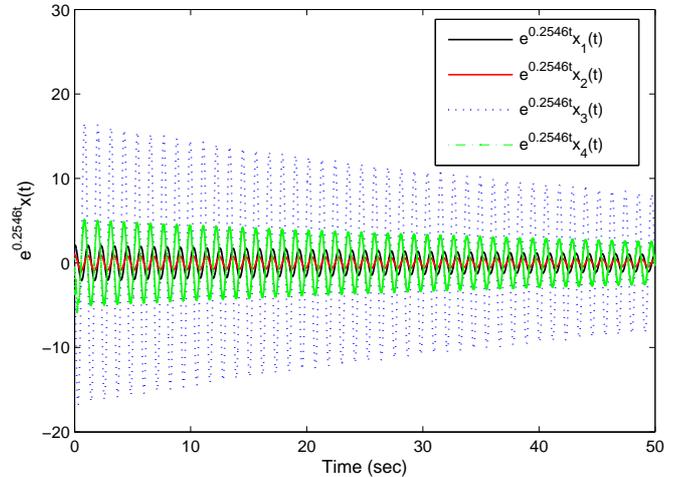}
\caption{Trajectory $e^{0.2546 t}x(t)$ with $h(t)=1+5|\sin(t)|$.}\label{fig:2}
\end{figure}

\begin{example}\rm
Consider system \eqref{e4.7} with
$$A=\begin{bmatrix} 0&1\\ -10 &-1\end{bmatrix},\quad
A_d=\begin{bmatrix}0&0.1\\ 0.1 &0.2\end{bmatrix}.$$

This example has been taken from \cite{PLL}. The obtained results by Corollary \ref{cr4.1} as listed in Table 4.
These results again show the effectiveness of our approach in proposed in this paper.

\begin{table*}[!htbp]\label{tb:4}
\caption{Upper bound of $h_2$ for various $h_1=1$}
\begin{tabular*}{\textwidth}{l @{\extracolsep{\fill}} llllll}\hline
$h_1$&0& 0.3& 0.7 & 1 & 2& NoDv\\ \hline
\cite{GGXH} &0.55&0.77&1.17&1.47&2.47&$50.5n^2+6.5n$\\
\cite{LLP} &1.35&1.64&2.02 &2.31 & 3.31& $9.5n^2+5.5n$ \\
\cite{PLL} &1.64&2.13&2.70&2.96&3.63&$21n^2+6n$\\
Corollary \ref{cr4.1} & 1.88&2.18&2.53&2.81&3.78& $10.5n^2+3.5n$\\
\hline
\end{tabular*}
\end{table*}
\end{example}

\section{Conclusion}

In this paper, new weighted integral inequalities (WIIs)
have been proposed. By employing WIIs,
new exponential stability criteria have been derived for
some classes of time-delay systems in the framework of
linear matrix inequalities. Numerous examples have been provided to
show the potential of WIIs and a large improvement
on the exponential convergence rate over the existing methods.

\end{document}